\documentclass[oneside,english,a4paper]{amsart}
\usepackage[T1]{fontenc}
\usepackage[latin9]{inputenc}
\usepackage{verbatim}
\usepackage{amsthm}
\usepackage{amstext}
\usepackage{amssymb}

\makeatletter
\numberwithin{equation}{section}
\numberwithin{figure}{section}
\numberwithin{table}{section}
\theoremstyle{plain}
\newtheorem{thm}{\protect\theoremname}[section]
  \theoremstyle{definition}
  \newtheorem{defn}[thm]{\protect\definitionname}
  \theoremstyle{remark}
  \newtheorem{rem}[thm]{\protect\remarkname}
  \theoremstyle{plain}
  \newtheorem{lem}[thm]{\protect\lemmaname}
  \theoremstyle{plain}
  \newtheorem{cor}[thm]{\protect\corollaryname}
  \theoremstyle{definition}
  \newtheorem{example}[thm]{\protect\examplename}

\usepackage{extarrow}
\usepackage{enumerate}
\usepackage{fontenc}
\usepackage[T1]{fontenc}
\usepackage{graphicx}
\usepackage[colorlinks=true, linkcolor=blue]{hyperref}
\usepackage{latexsym}
\usepackage{mathrsfs}
\usepackage{mathtext}
\usepackage{tikz}
\usepackage{textcomp}
\usepackage{upgreek}

\usetikzlibrary{arrows}
\usetikzlibrary{matrix}
\usetikzlibrary{shapes}
\usetikzlibrary{snakes}
\usetikzlibrary{matrix}

\DeclareMathOperator{\Add}{\textup{Add}}
\DeclareMathOperator{\Aff}{\textup{Aff}}
\DeclareMathOperator{\Alg}{\textup{Alg}}
\DeclareMathOperator{\Ann}{\textup{Ann}}
\DeclareMathOperator{\Arr}{\textup{Arr}}
\DeclareMathOperator{\Art}{\textup{Art}}
\DeclareMathOperator{\Ass}{\textup{Ass}}
\DeclareMathOperator{\Aut}{\textup{Aut}}
\DeclareMathOperator{\Autsh}{\underline{\textup{Aut}}}
\DeclareMathOperator{\Bi}{\textup{B}}
\DeclareMathOperator{\CAdd}{\textup{CAdd}}
\DeclareMathOperator{\CAlg}{\textup{CAlg}}
\DeclareMathOperator{\CMon}{\textup{CMon}}
\DeclareMathOperator{\CPMon}{\textup{CPMon}}
\DeclareMathOperator{\CRings}{\textup{CRings}}
\DeclareMathOperator{\CSMon}{\textup{CSMon}}
\DeclareMathOperator{\CaCl}{\textup{CaCl}}
\DeclareMathOperator{\Cart}{\textup{Cart}}
\DeclareMathOperator{\Cl}{\textup{Cl}}
\DeclareMathOperator{\Coh}{\textup{Coh}}
\DeclareMathOperator{\Coker}{\textup{Coker}}
\DeclareMathOperator{\Cov}{\textup{Cov}}
\DeclareMathOperator{\Der}{\textup{Der}}
\DeclareMathOperator{\Div}{\textup{Div}}
\DeclareMathOperator{\End}{\textup{End}}
\DeclareMathOperator{\Endsh}{\underline{\textup{End}}}
\DeclareMathOperator{\Ext}{\textup{Ext}}
\DeclareMathOperator{\Extsh}{\underline{\textup{Ext}}}
\DeclareMathOperator{\FAdd}{\textup{FAdd}}
\DeclareMathOperator{\FCoh}{\textup{FCoh}}
\DeclareMathOperator{\FGrad}{\textup{FGrad}}
\DeclareMathOperator{\FLoc}{\textup{FLoc}}
\DeclareMathOperator{\FMod}{\textup{FMod}}
\DeclareMathOperator{\FPMon}{\textup{FPMon}}
\DeclareMathOperator{\FRep}{\textup{FRep}}
\DeclareMathOperator{\FSMon}{\textup{FSMon}}
\DeclareMathOperator{\FVect}{\textup{FVect}}
\DeclareMathOperator{\Fibr}{\textup{Fibr}}
\DeclareMathOperator{\Fix}{\textup{Fix}}
\DeclareMathOperator{\Fl}{\textup{Fl}}
\DeclareMathOperator{\Fr}{\textup{Fr}}
\DeclareMathOperator{\Funct}{\textup{Funct}}
\DeclareMathOperator{\GAlg}{\textup{GAlg}}
\DeclareMathOperator{\GExt}{\textup{GExt}}
\DeclareMathOperator{\GHom}{\textup{GHom}}
\DeclareMathOperator{\GL}{\textup{GL}}
\DeclareMathOperator{\GMod}{\textup{GMod}}
\DeclareMathOperator{\GRis}{\textup{GRis}}
\DeclareMathOperator{\GRiv}{\textup{GRiv}}
\DeclareMathOperator{\Gal}{\textup{Gal}}
\DeclareMathOperator{\Gl}{\textup{Gl}}
\DeclareMathOperator{\Grad}{\textup{Grad}}
\DeclareMathOperator{\Hilb}{\textup{Hilb}}
\DeclareMathOperator{\Hl}{\textup{H}}
\DeclareMathOperator{\Hom}{\textup{Hom}}
\DeclareMathOperator{\Homsh}{\underline{\textup{Hom}}}
\DeclareMathOperator{\ISym}{\textup{Sym}^*}
\DeclareMathOperator{\Imm}{\textup{Im}}
\DeclareMathOperator{\Irr}{\textup{Irr}}
\DeclareMathOperator{\Iso}{\textup{Iso}}
\DeclareMathOperator{\Isosh}{\underline{\textup{Iso}}}
\DeclareMathOperator{\Ker}{\textup{Ker}}
\DeclareMathOperator{\LAdd}{\textup{LAdd}}
\DeclareMathOperator{\LAlg}{\textup{LAlg}}
\DeclareMathOperator{\LMon}{\textup{LMon}}
\DeclareMathOperator{\LPMon}{\textup{LPMon}}
\DeclareMathOperator{\LRings}{\textup{LRings}}
\DeclareMathOperator{\LSMon}{\textup{LSMon}}
\DeclareMathOperator{\Left}{\textup{L}}
\DeclareMathOperator{\Lex}{\textup{Lex}}
\DeclareMathOperator{\Loc}{\textup{Loc}}
\DeclareMathOperator{\M}{\textup{M}}
\DeclareMathOperator{\ML}{\textup{ML}}
\DeclareMathOperator{\MLex}{\textup{MLex}}
\DeclareMathOperator{\Map}{\textup{Map}}
\DeclareMathOperator{\Mod}{\textup{Mod}}
\DeclareMathOperator{\Mon}{\textup{Mon}}
\DeclareMathOperator{\Ob}{\textup{Ob}}
\DeclareMathOperator{\Obj}{\textup{Obj}}
\DeclareMathOperator{\PDiv}{\textup{PDiv}}
\DeclareMathOperator{\PGL}{\textup{PGL}}
\DeclareMathOperator{\PML}{\textup{PML}}
\DeclareMathOperator{\PMLex}{\textup{PMLex}}
\DeclareMathOperator{\PMon}{\textup{PMon}}
\DeclareMathOperator{\Pic}{\textup{Pic}}
\DeclareMathOperator{\Picsh}{\underline{\textup{Pic}}}
\DeclareMathOperator{\Pro}{\textup{Pro}}
\DeclareMathOperator{\Proj}{\textup{Proj}}
\DeclareMathOperator{\QAdd}{\textup{QAdd}}
\DeclareMathOperator{\QAlg}{\textup{QAlg}}
\DeclareMathOperator{\QCoh}{\textup{QCoh}}
\DeclareMathOperator{\QMon}{\textup{QMon}}
\DeclareMathOperator{\QPMon}{\textup{QPMon}}
\DeclareMathOperator{\QRings}{\textup{QRings}}
\DeclareMathOperator{\QSMon}{\textup{QSMon}}
\DeclareMathOperator{\R}{\textup{R}}
\DeclareMathOperator{\Rep}{\textup{Rep}}
\DeclareMathOperator{\Rings}{\textup{Rings}}
\DeclareMathOperator{\Riv}{\textup{Riv}}
\DeclareMathOperator{\SFibr}{\textup{SFibr}}
\DeclareMathOperator{\SMLex}{\textup{SMLex}}
\DeclareMathOperator{\SMex}{\textup{SMex}}
\DeclareMathOperator{\SMon}{\textup{SMon}}
\DeclareMathOperator{\SchI}{\textup{SchI}}
\DeclareMathOperator{\Sh}{\textup{Sh}}
\DeclareMathOperator{\Soc}{\textup{Soc}}
\DeclareMathOperator{\Spec}{\textup{Spec}}
\DeclareMathOperator{\Specsh}{\underline{\textup{Spec}}}
\DeclareMathOperator{\Stab}{\textup{Stab}}
\DeclareMathOperator{\Supp}{\textup{Supp}}
\DeclareMathOperator{\Sym}{\textup{Sym}}
\DeclareMathOperator{\TMod}{\textup{TMod}}
\DeclareMathOperator{\Top}{\textup{Top}}
\DeclareMathOperator{\Tor}{\textup{Tor}}
\DeclareMathOperator{\Vect}{\textup{Vect}}
\DeclareMathOperator{\alt}{\textup{ht}}
\DeclareMathOperator{\car}{\textup{char}}
\DeclareMathOperator{\codim}{\textup{codim}}
\DeclareMathOperator{\degtr}{\textup{degtr}}
\DeclareMathOperator{\depth}{\textup{depth}}
\DeclareMathOperator{\divis}{\textup{div}}
\DeclareMathOperator{\et}{\textup{et}}
\DeclareMathOperator{\ffpSch}{\textup{ffpSch}}
\DeclareMathOperator{\h}{\textup{h}}
\DeclareMathOperator{\ilim}{\displaystyle{\lim_{\longrightarrow}}}
\DeclareMathOperator{\ind}{\textup{ind}}
\DeclareMathOperator{\indim}{\textup{inj dim}}
\DeclareMathOperator{\lf}{\textup{LF}}
\DeclareMathOperator{\op}{\textup{op}}
\DeclareMathOperator{\ord}{\textup{ord}}
\DeclareMathOperator{\pd}{\textup{pd}}
\DeclareMathOperator{\plim}{\displaystyle{\lim_{\longleftarrow}}}
\DeclareMathOperator{\pr}{\textup{pr}}
\DeclareMathOperator{\pt}{\textup{pt}}
\DeclareMathOperator{\rk}{\textup{rk}}
\DeclareMathOperator{\tr}{\textup{tr}}
\DeclareMathOperator{\type}{\textup{r}}
\DeclareMathOperator*{\colim}{\textup{colim}}

\usepackage[colorinlistoftodos]{todonotes}
\usepackage{enumitem}
\setenumerate[1]{label={\arabic*)}} 
\usepackage[a4paper]{geometry}
\renewcommand{\Lex}{\textup{Mon}}

\theoremstyle{plain}
\newtheorem{thma}{}

\author[Tonini]{Fabio Tonini}

\address[Tonini]{Freie Universit\"at Berlin\\
    FB Mathematik und Informatik\\
    Arnimallee 3\\ Zimmer 112A\\
    14195 Berlin\\ Deutschland}
\email{tonini@zedat.fu-berlin.de}

\makeatother

\usepackage{babel}
  \providecommand{\corollaryname}{Corollary}
  \providecommand{\definitionname}{Definition}
  \providecommand{\examplename}{Example}
  \providecommand{\lemmaname}{Lemma}
  \providecommand{\remarkname}{Remark}
\providecommand{\theoremname}{Theorem}

\begin{document}

\title{Trace map and regularity of finite extensions of a DVR}

\maketitle
\global\long\def\A{\mathbb{A}}

\global\long\def\Ab{(\textup{Ab})}

\global\long\def\C{\mathbb{C}}

\global\long\def\Cat{(\textup{cat})}

\global\long\def\Di#1{\textup{D}(#1)}

\global\long\def\E{\mathcal{E}}

\global\long\def\F{\mathbb{F}}

\global\long\def\GCov{G\textup{-Cov}}

\global\long\def\Gcat{(\textup{Galois cat})}

\global\long\def\Gfsets#1{#1\textup{-fsets}}

\global\long\def\Gm{\mathbb{G}_{m}}

\global\long\def\GrCov#1{\textup{D}(#1)\textup{-Cov}}

\global\long\def\Grp{(\textup{Grps})}

\global\long\def\Gsets#1{(#1\textup{-sets})}

\global\long\def\HCov{H\textup{-Cov}}

\global\long\def\MCov{\textup{D}(M)\textup{-Cov}}

\global\long\def\MHilb{M\textup{-Hilb}}

\global\long\def\N{\mathbb{N}}

\global\long\def\PGor{\textup{PGor}}

\global\long\def\PGrp{(\textup{Profinite Grp})}

\global\long\def\PP{\mathbb{P}}

\global\long\def\Pj{\mathbb{P}}

\global\long\def\Q{\mathbb{Q}}

\global\long\def\RCov#1{#1\textup{-Cov}}

\global\long\def\RR{\mathbb{R}}

\global\long\def\Sch{\textup{Sch}}

\global\long\def\WW{\textup{W}}

\global\long\def\Z{\mathbb{Z}}

\global\long\def\acts{\curvearrowright}

\global\long\def\alA{\mathscr{A}}

\global\long\def\alB{\mathscr{B}}

\global\long\def\arr{\longrightarrow}

\global\long\def\arrdi#1{\xlongrightarrow{#1}}

\global\long\def\catC{\mathscr{C}}

\global\long\def\catD{\mathscr{D}}

\global\long\def\catF{\mathscr{F}}

\global\long\def\catG{\mathscr{G}}

\global\long\def\comma{,\ }

\global\long\def\covU{\mathcal{U}}

\global\long\def\covV{\mathcal{V}}

\global\long\def\covW{\mathcal{W}}

\global\long\def\duale#1{{#1}^{\vee}}

\global\long\def\fasc#1{\widetilde{#1}}

\global\long\def\fsets{(\textup{f-sets})}

\global\long\def\iL{r\mathscr{L}}

\global\long\def\id{\textup{id}}

\global\long\def\la{\langle}

\global\long\def\odi#1{\mathcal{O}_{#1}}

\global\long\def\ra{\rangle}

\global\long\def\set{(\textup{Sets})}

\global\long\def\sets{(\textup{Sets})}

\global\long\def\shA{\mathcal{A}}

\global\long\def\shB{\mathcal{B}}

\global\long\def\shC{\mathcal{C}}

\global\long\def\shD{\mathcal{D}}

\global\long\def\shE{\mathcal{E}}

\global\long\def\shF{\mathcal{F}}

\global\long\def\shG{\mathcal{G}}

\global\long\def\shH{\mathcal{H}}

\global\long\def\shI{\mathcal{I}}

\global\long\def\shJ{\mathcal{J}}

\global\long\def\shK{\mathcal{K}}

\global\long\def\shL{\mathcal{L}}

\global\long\def\shM{\mathcal{M}}

\global\long\def\shN{\mathcal{N}}

\global\long\def\shO{\mathcal{O}}

\global\long\def\shP{\mathcal{P}}

\global\long\def\shQ{\mathcal{Q}}

\global\long\def\shR{\mathcal{R}}

\global\long\def\shS{\mathcal{S}}

\global\long\def\shT{\mathcal{T}}

\global\long\def\shU{\mathcal{U}}

\global\long\def\shV{\mathcal{V}}

\global\long\def\shW{\mathcal{W}}

\global\long\def\shX{\mathcal{X}}

\global\long\def\shY{\mathcal{Y}}

\global\long\def\shZ{\mathcal{Z}}

\global\long\def\st{\ | \ }

\global\long\def\stA{\mathcal{A}}

\global\long\def\stB{\mathcal{B}}

\global\long\def\stC{\mathcal{C}}

\global\long\def\stD{\mathcal{D}}

\global\long\def\stE{\mathcal{E}}

\global\long\def\stF{\mathcal{F}}

\global\long\def\stG{\mathcal{G}}

\global\long\def\stH{\mathcal{H}}

\global\long\def\stI{\mathcal{I}}

\global\long\def\stJ{\mathcal{J}}

\global\long\def\stK{\mathcal{K}}

\global\long\def\stL{\mathcal{L}}

\global\long\def\stM{\mathcal{M}}

\global\long\def\stN{\mathcal{N}}

\global\long\def\stO{\mathcal{O}}

\global\long\def\stP{\mathcal{P}}

\global\long\def\stQ{\mathcal{Q}}

\global\long\def\stR{\mathcal{R}}

\global\long\def\stS{\mathcal{S}}

\global\long\def\stT{\mathcal{T}}

\global\long\def\stU{\mathcal{U}}

\global\long\def\stV{\mathcal{V}}

\global\long\def\stW{\mathcal{W}}

\global\long\def\stX{\mathcal{X}}

\global\long\def\stY{\mathcal{Y}}

\global\long\def\stZ{\mathcal{Z}}

\global\long\def\then{\ \Longrightarrow\ }

\global\long\def\L{\textup{L}}

\global\long\def\l{\textup{l}}

\begin{abstract}
We interpret the regularity of a finite and flat extension of a discrete
valuation ring in terms of the trace map of the extension.
\end{abstract}

\section*{Introduction}

Let $R$ be a ring and $A$ be an $R$-algebra which is projective
and finitely generated as $R$-module. We denote by $\tr_{A/R}\colon A\arr R$
the trace map and by $\widetilde{\tr}_{A/R}\colon A\arr A^{\vee}=\Hom_{R}(A,R)$
the map $a\longmapsto\tr_{A/R}(a\cdot-)$. It is a well known result
of commutative algebra that the étaleness of the extension $A/R$
is entirely encoded in the trace map $\tr_{A/R}$: $A/R$ is étale
if and only if the map $\widetilde{\tr}_{A/R}\colon A\arr\duale A$
is an isomorphism (see \cite[Proposition 4.10]{SGA1}). In this case,
if $R$ is a DVR (discrete valuation ring) it follows that $A$ is
regular (that is a product of Dedekind domains), while the converse
is clearly not true because extensions of Dedekind domains are often
ramified.

In this paper we show how to read the regularity of $A$ in terms
of the trace map $\tr_{A/R}$. In order to express our result we need
some notations and definitions. Let us assume from now on that the
ring $R$ is a DVR with residue field $k_{R}$. We first extend the
notion of tame extensions and ramification index: given a maximal
ideal $p$ of $A$ we set
\[
e(p,A/R)=\frac{\dim_{k_{R}}(A_{p}\otimes_{R}k_{R})}{[k(p):k_{R}]}
\]
where $k(p)=A/p$, and we call it the \emph{ramification index }of
$p$ in the extension $A/R$. Notice that $e(p,A/R)\in\N$ (see \ref{lem:ramification indexes and ranks}).
We say that $A/R$ is \emph{tame} (over the maximal ideal of $R$)
if the ramification indexes of all maximal ideals of $A$ are coprime
with $\car k_{R}$. Those definitions agree with the usual ones when
$A$ is a Dedekind domain. We also set
\[
\shQ_{A/R}=\Coker(A\arrdi{\widetilde{\tr}_{A/R}}\duale A)\comma f^{A/R}=\l(\shQ_{A/R})
\]
where $\l$ denotes the length function. Alternatively $f^{A/R}$
can be seen as the valuation of the discriminant section $\det\widetilde{\tr}_{A/R}$.
We also denote by $|\Spec(A\otimes_{R}\overline{k_{R}})|$ the number
of primes of $A\otimes_{R}\overline{k_{R}}$: this number can also
be computed as
\[
|\Spec(A\otimes_{R}\overline{k_{R}})|=\sum_{p\text{ maximal ideals of }A}[F_{p}:k_{R}]
\]
where $F_{p}$ denotes the maximal separable extension of $k_{R}$
inside $k(p)=A/p$ (see \ref{cor:Length of geometric fiber}). Finally
we will say that $A/R$ has separable residue fields (over the maximal
ideal of $R$) if for all maximal ideals $p$ of $A$ the finite extension
$k(p)/k_{R}$ is separable. The theorem we are going to prove is the
following:

\begin{thma}\label{main}
Let $R$ be a DVR and $A$ be a finite and flat $R$-algebra. Then we have the inequality
$$
f^{A/R}\geq\rk A-|\Spec(A\otimes_{R}\overline{k_{R}})|
$$
and the following conditions are equivalent:
\begin{enumerate}
\item the equality holds in the inequality above;
\item  $A$ is regular and $A/R$ is tame with separable residue fields;
\item $A/R$ is tame with separable residue fields and the $R$-module $\shQ_{A/R}$ is defined over $k_R$, that is $m_R\shQ_{A/R}=0$, where $m_R$ denotes the maximal ideal of $R$.
\end{enumerate}
\end{thma}

The implication $2)\then1)$ is classical, because it follows directly
from the relation between the different and ramification indexes (see
\cite[III, §6, Proposition 13]{Serre1979}). Notice that the étaleness
of $A/R$ is equivalent to any of the following conditions: $f^{A/R}=0$,
$\rk A=|\Spec(A\otimes_{R}\overline{k_{R}})|$, $\shQ_{A/R}=0$.

The above result has already been proved in my Ph.D. thesis \cite[Theorem 4.4.4]{Tonini2013a}
under the assumption that a finite solvable group $G$ acts on $A$
in a way that $A^{G}=R$ and using a completely different strategy:
using induction and finding a filtration of $R$-algebras of $R\subseteq A$
starting from a filtration of normal subgroups of $G$. \ref{main}
is an essential ingredient in the proof of \cite[Theorem C]{Tonini2015}
which generalizes \cite[Theorem 4.4.7]{Tonini2013a}.

The paper is organized as follows. In the first section we discuss
the notion of tameness and ramification index and recall some basic
facts of commutative algebra and, in particular, about the trace map.
In the second section we prove the \ref{main}.

\section*{Notation}

All rings and algebras in this paper are commutative with unity. 

Given a ring $R$ and a prime $q$ we denote by $k(q)$ the residue
field of $R_{q}$. If $A$ is a finite $R$-algebra we say that $A/R$
has separable residue fields over a prime $q$ of $R$ if for all
primes $p$ of $A$ lying over $q$ the finite extension of fields
$k(p)/k(q)$ is separable. We say that $A/R$ has separable residue
fields if it has this property over all primes of $R$.

The following conditions are equivalent (see \cite[Proposition 1.4.7]{EGAIV-1}):
$A/R$ is finite, flat and finitely presented; $A$ is finitely generated
and projective as $R$-module; $A$ is finitely presented as $R$-module
and for all primes $q$ of $R$ the $R_{q}$-module $A_{q}$ is free.
In this situation there is a well-defined trace function $\Hom_{R}(A,A)\arr R$
which extends the usual trace of matrices and commutes with arbitrary
extensions of scalars. The trace map of $A/R$, denoted by $\tr_{A/R}\colon A\arr R$
(or simply $\tr_{A}$), is the composition $A\arr\Hom_{R}(A,A)\arr R$,
where the first map is induced by the multiplication in $A$.

If $R$ is a local ring we denote by $m_{R}$ its maximal ideal and
by $k_{R}$ its residue field. A DVR is a discrete valuation ring.

If $k$ is a field we denote by $\overline{k}$ an algebraic closure
of $k$ and by $k^{s}$ a separable closure of $k$.

\section*{Acknowledgements}

I would like to thank Angelo Vistoli and Dajano Tossici for all the
useful conversations we had and all the suggestions I received.

\section{Preliminaries on tameness and trace map}

We fix a ring $R$ and a finite, flat and finitely presented $R$-algebra
$A$ over $R$, that is an $R$-algebra $A$ which is finitely presented
as $R$-module and such that $A_{q}$ is a free $R_{q}$-module of
finite rank for all prime $q$ of $R$.
\begin{defn}
Given a prime ideal $p$ of $A$ lying above the prime $q$ of $R$
we set
\[
e(p,A/R)=\frac{\dim_{k(q)}(A_{p}\otimes_{R_{q}}k(q))}{[k(p):k(q)]}
\]
and we call it the \emph{ramification index }of $p$ in the extension
$A/R$. We say that $A/R$ is \emph{tame }in $p\in\Spec A$ if $e(p,A/R)$
is coprime with $\car k(q)$, is \emph{tame }over $q\in\Spec R$ if
it is tame over all primes of $A$ over $q$ and, finally, we say
that it is \emph{tame }if it is tame over all primes of $R$ (or in
all primes of $A$).\end{defn}
\begin{rem}
The number $e(p,A/R)$ is always a natural number as shown below.
Moreover if $R$ and $A$ are Dedekind domains, the notion of ramification
index agrees with the usual one. \end{rem}
\begin{lem}
\label{lem:ramification indexes and ranks} Let $p$ be a prime of
$A$ lying over the prime $q$ of $R$ and denote by $F$ the maximal
separable extension of $k(q)$ inside $k(p)$. Then $e(p,A/R)$ is
a natural number and
\[
A_{p}\otimes_{R_{q}}\overline{k(q)}\simeq B_{1}\times\cdots\times B_{s}\text{ with }B_{i}\text{ local, }\dim_{\overline{k(q)}}B_{i}=e(p,A/R)[k(p):F]\text{ and }s=[F:k(q)]
\]
\end{lem}
\begin{proof}
We can assume that $R=k(q)=k$ is a field and that $A$ is local.
Set also $L=k(p)$ and write
\[
L\otimes_{k}\overline{k}\simeq C_{1}\times\cdots\times C_{r}\comma A\otimes_{k}\overline{k}\simeq B_{1}\times\cdots\times B_{s}
\]
for the decompositions into local rings. In particular we have that
$r=s$ because $A\otimes_{k}\overline{k}\arr L\otimes_{k}\overline{k}$
is surjective with nilpotent kernel. Moreover this map splits as a
product of surjective maps $B_{i}\arr C_{i}$. Denote by $J_{i}$
their kernels and by $\phi\colon A\otimes_{k}\overline{k}\arr B_{1}\times\cdots\times B_{s}$
the previous isomorphism of rings. For all $t\in\N$ (including $t=0$)
we have 
\[
\phi((m_{A}\otimes_{k}\overline{k})^{t})=J_{1}^{t}\times\cdots\times J_{s}^{t}
\]
In particular for all $t\in\N$ we have
\[
\left(\frac{J_{1}^{t}}{J_{1}^{t+1}}\right)\times\cdots\times\left(\frac{J_{s}^{t}}{J_{s}^{t+1}}\right)\simeq\frac{(m_{A}\otimes_{k}\overline{k})^{t}}{(m_{A}\otimes_{k}\overline{k})^{t+1}}\simeq\frac{m_{A}^{t}}{m_{A}^{t+1}}\otimes_{k}\overline{k}\simeq(L\otimes_{k}\overline{k})^{f(t)}\simeq C_{1}^{f(t)}\times\cdots\times C_{s}^{f(t)}
\]
as $L\otimes_{k}\overline{k}$-modules, where $f(t)=\dim_{L}(m_{A}^{t}/m_{A}^{t+1})$.
We can conclude that
\[
\left(\frac{J_{i}^{t}}{J_{i}^{t+1}}\right)\simeq C_{i}^{f(t)}\text{ for all }i,t
\]
In particular
\[
\dim_{\overline{k}}B_{i}=\sum_{t\in\N}\dim_{\overline{k}}\left(\frac{J_{i}^{t}}{J_{i}^{t+1}}\right)=(\dim_{\overline{k}}C_{i})\sum_{t\in\N}f(t)
\]
because $J_{i}$ is nilpotent. Similarly we have
\[
\dim_{k}A=\sum_{t\in N}\dim_{k}\left(\frac{m_{A}^{t}}{m_{A}^{t+1}}\right)=[L:k]\sum_{t\in\N}f(t)
\]
In particular $[L:k]\mid\dim_{k}A$, so that the ramification index
is a natural number. By a direct computation we see that everything
follows if we show that $\dim_{\overline{k}}C_{i}=[L:F]$. Since $F/k$
is separable we know that $F\otimes_{k}\overline{k}\simeq\overline{k}^{[F:k]}$.
Each factor corresponds to an embedding $\sigma\colon F\arr\overline{k}$
such that $\sigma_{|k}=\id_{k}$ and
\[
L\otimes_{k}\overline{k}\simeq L\otimes_{F}(F\otimes_{k}\overline{k})\simeq\prod_{\sigma\in\Hom_{k}(F,\overline{k})}L\otimes_{F,\sigma}\overline{k}
\]
This is exactly the decomposition into local rings because, since
$L/F$ is purely inseparable, all the rings $L\otimes_{F,\sigma}\overline{k}$
are local. This ends the proof.\end{proof}
\begin{cor}
\label{cor:Length of geometric fiber} Let $q$ be a prime of $R$.
Then
\[
|\Spec(A\otimes_{R}\overline{k(q)})|=\sum_{\begin{array}{c}
p\text{ primes of}\\
A\text{ over }q
\end{array}}[F_{p}:k(q)]
\]
where $F_{p}$ denotes the maximal separable extension of $k(q)$
inside $k(p)$.\end{cor}
\begin{proof}
It follows from \ref{lem:ramification indexes and ranks} using the
fact that $A\otimes_{R}k(q)$ is the product of the $A_{p}\otimes_{R_{q}}k(q)$
for $p$ running through all primes of $A$ over $q$. \end{proof}
\begin{defn}
Let $p$ be a prime of $A$ lying over the prime $q$ of $R$. We
denote by $h(p,A/R)$ the common length of the localizations of $A_{p}\otimes_{R_{q}}\overline{k(q)}$,
that is, following notation from \ref{lem:ramification indexes and ranks},
$h(p,A/R)=\dim_{k}B_{i}=e(p,A/R)[k(p):F]$.\end{defn}
\begin{lem}
\label{lem:invariance of h} Let $p$ be a prime of $A$, $R'$ be
an $R$-algebra and $p'$ be a prime of $A'=A\otimes_{R}R'$ over
the prime $p$. Then
\[
h(p,A/R)=h(p',A'/R')
\]
\end{lem}
\begin{proof}
We can assume that $R=k$ and $R'=k'$ are fields and that $A$ is
local. Moreover by definition of the function $h(-)$ we can also
assume that $k$ and $k'$ are algebraically closed. In this case
$h(p,A/k)=\dim_{k}A$ and, since $A'=A\otimes_{k}k'$ is again local,
$h(p',A'/k')=\dim_{k'}A'=\dim_{k}A$.\end{proof}
\begin{cor}
\label{cor:tameness and h} Let $p$ be a prime of $A$. Then $A/R$
is tame in $p$ and $k(p)/k(q)$ is separable if and only if the number
$h(p,A/R)$ is coprime with $\car k(q)$.\end{cor}
\begin{proof}
We can assume that $R=k=k(q)$ is a field and $A$ is local with residue
field $L$. Let also $F$ be the maximal separable extension of $k$
inside $L$. Thanks to \ref{lem:ramification indexes and ranks},
the last condition in the statement is that the number $e(m_{A},A/k)[L:F]$
is coprime with $\car k$. Since $L/F$ is purely inseparable, so
that $[L:F]$ is either $1$ or a power of $\car k$, this is the
same as $ $$A/k$ being tame and $L=F$, that is $L/k$ is separable.\end{proof}
\begin{rem}
Let $q$ be a prime of $R$, $R'$ be an $R$-algebra and $q'$ be
a prime of $R'$ over $R$. By \ref{lem:invariance of h} and \ref{cor:tameness and h}
it follows that $A/R$ is tame over $q$ and $A\otimes_{R}k(q)$ has
separable residue fields if and only if the same is true for the extension
$(A\otimes_{R}R')/R'$ with respect to the prime $q'$. On the other
hand tameness alone does not satisfy the same base change property,
and, in particular, the function $e(-)$ does not satisfy Lemma \ref{lem:invariance of h}.
The counterexample is a finite purely inseparable extension $L/k$:
we have that $e(m_{k},L/k)=1$, so that $L/k$ is tame, while $e(m_{\overline{k}},L\otimes_{k}\overline{k}/\overline{k})=[L:k]$
because $L\otimes_{k}\overline{k}$ is local, so that $L\otimes_{k}\overline{k}/k$
is not tame.\end{rem}
\begin{lem}
\label{lem:computing tr over a field} Assume that $R=k$ is a field,
that $A$ is local and set $\pi\colon A\arr k_{A}$ for the projection.
Then
\[
\tr_{A/k}=e(m_{A},A/k)\tr_{k_{A}/k}\circ\pi
\]
\end{lem}
\begin{proof}
Set $P=m_{A}$, $L=k_{A}$ and let $x_{1,i},\dots,x_{r_{i},i}\in P^{i}$
be elements whose projections form an $L$-basis of $P^{i}/P^{i+1}$.
We set $x_{1,0}=1$. Let also $y_{1},\dots,y_{s}\in A$ be elements
whose projections form a $k$-basis of $L$, where $s=\dim_{k}L$.
It is easy to see that the collection
\[
\{x_{\alpha,i}y_{\beta}\}_{1\leq\alpha\leq r_{i},1\leq\beta\leq s}
\]
is a $k$-basis of $P^{i}/P^{i+1}$ for all $i\geq0$. By an inverse
induction starting from the nilpotent index of $P$, it also follows
that
\[
\shB_{n}=\{x_{\alpha,i}y_{\beta}\}_{1\leq\alpha\leq r_{i},1\leq\beta\leq s,i\geq n}
\]
is a $k$-basis of $P^{n}$ for all $n\geq0$. In particular, when
$n=0$ we get a $k$-basis of $A=P^{0}$.

We are going to compute the trace map $\tr_{A}$ over the basis $\shB_{0}$.
Consider an index $i>0$. For all possible $\alpha,\beta,\gamma,\delta,j$
we have that
\[
z=(x_{\alpha,i}y_{\beta})(x_{\gamma,j}y_{\delta})\in P^{i+j}\subseteq P^{j+1}
\]
Thus $z$ is a linear combination of vectors in $\shB_{j+1}$, which
does not contain $(x_{\gamma,j}y_{\delta})$. It follows that $\tr_{A}(x_{\alpha,i}y_{\beta})=0$
for all $i>0$, that is $\tr_{A}(P)=0$ which agrees with the formula
in the statement. It remains to compute $\tr_{A}(y_{\beta})$. Write
\[
y_{\beta}y_{\delta}=\sum_{q}b_{\beta,\delta,q}y_{q}+u_{\beta,\delta}\text{ with \ensuremath{u_{\beta,\delta}\in P}\text{ and }}b_{\beta,\delta,q}\in k
\]
It follows that 
\[
\tr_{L/k}(\pi(y_{\beta}))=\sum_{\delta}b_{\beta,\delta,\delta}
\]
Let's multiply now $y_{\beta}$ with an element $x_{\alpha,i}y_{\delta}$,
obtaining
\[
z=y_{\beta}(x_{\alpha,i}y_{\delta})=x_{\alpha,i}u_{\beta,\delta}+\sum_{q}b_{\beta,\delta,q}x_{\alpha,i}y_{q}
\]
If $i=0$, so that $\alpha=1$ and $x_{1,0}=1$, the coefficient of
$z$ with respect to $y_{\delta}$ is $b_{\beta,\delta,\delta}$ because
$u_{\beta,\delta}\in P$. If $i>0$ then the coefficient of $z$ with
respect to $(x_{\alpha,i}y_{\delta})$ is again $b_{\beta,\delta,\delta}$
because $x_{\alpha,i}u_{\beta,\delta}\in P^{i+1}$ and thus can be
written using only the vectors in $\shB_{i+1}$. In conclusion we
have that
\[
\tr_{A}(y_{\beta})=\sum_{i,\alpha,\delta}b_{\beta,\delta,\delta}=(\sum_{\delta}b_{\beta,\delta,\delta})(\sum_{\alpha,i}1)=\tr_{L}(\pi(y_{\beta}))C\text{ with }C=(\sum_{i,\alpha}1)
\]
Thus $\tr_{A}=C\tr_{L}\circ\pi$ and, finally,
\[
C=(\sum_{i,\alpha}1)=\sum_{i\geq0}\dim_{L}(\frac{P^{i}}{P^{i+1}})=\frac{\dim_{k}A}{[L:k]}=e(m_{A},A/k)
\]
\end{proof}
\begin{cor}
\label{cor:ker of the trace} Assume that $R$ and $A$ are local.
Then\end{cor}
\begin{enumerate}
\item $\tr_{A/R}(m_{A})\subseteq m_{R}$;
\item if $k_{A}=k_{R}$ and $\rk A\in R^{*}$ then $\Ker\tr_{A/R}\subseteq m_{A}$.\end{enumerate}
\begin{proof}
Since $\tr_{A/R}\otimes_{R}k_{R}=\tr_{(A\otimes_{R}k_{R})/k_{R}}$
point $1)$ follows from \ref{lem:computing tr over a field} because
$ $$\tr_{(A\otimes_{R}k_{R})/k_{R}}(m_{A\otimes_{R}k_{R}})=0$. Assume
now the hypothesis of $2)$ and let $x\in\Ker\tr_{A}$. If $x\notin m_{A}$
there exists $\lambda\in R^{*}$ such that $y=x-\lambda\in m_{A}$,
so that
\[
\tr_{A}(x)=0=\rk A\lambda+\tr_{A}(y)\in R^{*}+m_{R}=R^{*}
\]
which is impossible.\end{proof}
\begin{lem}
\label{lem:when tr is surjective} If $R$ and $A$ are local then
\[
\tr_{A/R}\colon A\arr R\text{ is surjective}\iff h(m_{A},A/R)\text{ and }\car k_{R}\text{ are coprime}
\]
\end{lem}
\begin{proof}
By Nakayama's lemma $\tr_{A/R}\colon A\arr R$ is surjective if and
only if $\tr_{A/R}\otimes_{R}k_{R}=\tr_{A\otimes k_{R}/k_{R}}\colon A\otimes k_{R}\arr k_{R}$
is so. Thus we can assume that $R=k$ is a field. By \ref{lem:computing tr over a field}
$\tr_{A/k}$ is surjective if and only if $\tr_{k_{A}/k}$ is surjective,
i.e. $k_{A}/k$ is separable, and $e(m_{A},A/k)\in k^{*}$. The result
then follows from \ref{cor:tameness and h}.
\end{proof}

\section{Regularity of finite extensions of DVR}

We fix a DVR $R$ and a finite and flat $R$-algebra $A$, so that
$A$ is free of finite rank $\rk A$ as $R$-module. We will use the
following notation

\[
\widetilde{\tr}_{A/R}\colon A\arr\duale A\comma a\longmapsto\tr_{A/R}(a\cdot-)
\]
\[
\shQ_{A/R}=\Coker(A\arrdi{\widetilde{\tr}_{A/R}}\duale A)\comma f^{A/R}=\l(\shQ_{A/R})
\]
where $\l$ denotes the length function. For simplicity we will replace
$A/R$ with $A$ if no confusion can arise. %

\begin{rem}
\label{rem: first considerations on traces} By standard arguments
we have that $f^{A}$ coincides with the valuation over $R$ of $\det(\widetilde{\tr}_{A})$.
Moreover the following conditions are equivalent:
\begin{enumerate}
\item $A$ is generically étale over $R$;
\item $f^{A}<\infty$;
\item $\tr_{A}\colon A\arr\duale A$ is injective.
\end{enumerate}
In particular we see that all three conditions in \ref{main} implies
that $A/R$ is generically étale. This also means that $A/R$ is tame
with separable residue fields if and only if $A\otimes_{R}k_{R}/k_{R}$,
or $A\otimes_{R}\overline{k_{R}}/\overline{k_{R}}$, is so.

If $\beta=\{x_{0},\dots,x_{n}\}$ is an $R$-basis of $A$ then the
matrix of $\widetilde{\tr}_{A}\colon A\arr\duale A$ with respect
to $\beta$ and its dual is given by $T=(\tr_{A}(x_{i}x_{j}))_{i,j}$.
In particular we can conclude that, if $\tr_{A}\colon A\arr R$ is
not surjective, then $f^{A}\geq\rk A$, because all entries of $T$
belongs to $m_{R}$. If $\tr_{A}(1)=\rk A\in R^{*}$ we have a decomposition
$A=R1\oplus\Ker\tr_{A}$. In particular if $x_{0}=1$ and $x_{1},\dots,x_{n}\in\Ker\tr_{A}$
(such a basis always exists locally) the matrix $T$ has the form
\[
\left(\begin{array}{cc}
\rk A & 0\\
0 & N
\end{array}\right)
\]
and therefore $\shQ_{A}\simeq\Coker N$.
\end{rem}

\begin{rem}
\label{rem: passing to the strict henselization} Let $R^{s}$ be
the strict Henselization of $R$ and set $A^{s}=A\otimes_{R}R^{s}$.
In particular $R^{s}$ is a discrete valuation ring with residue field
the separable closure $k_{R}^{s}$ of $k_{R}$. Using that the extension
$R\arr R^{s}$ is faithfully flat and unramified, it is easy to see
that 
\[
f^{A/R}=f^{A^{s}/R^{s}}\comma|\Spec(A\otimes_{R}\overline{k_{R}})|=|\Spec(A^{s}\otimes_{R^{s}}\overline{k_{R^{s}}})|
\]
that $\shQ_{A/R}$ is defined over $k_{R}$ if and only if $\shQ_{A^{s}/R^{s}}$
is defined over $k_{R^{s}}$ and that $A/R$ is tame with separable
residue fields if and only if $A^{s}/R^{s}$ is so.

Since $R^{s}$ is Henselian, we have a decomposition $A^{s}=A_{1}\times\cdots\times A_{q}$
where the $A_{j}$ are local rings finite and flat over $R^{s}$.
Moreover $A_{j}\otimes_{R^{s}}\overline{k_{R^{s}}}$ are still local,
so that $q=|\Spec(A\otimes_{R}\overline{k_{R}})|$. Since $\widetilde{\tr}_{A^{s}/R^{s}}\colon A^{s}\arr\duale{(A^{s})}$
is the direct sum of the $\widetilde{\tr}_{A_{j}/R^{s}}\colon A_{j}\arr\duale{(A_{j})}$
we have
\[
\shQ_{A^{s}/R^{s}}\simeq\bigoplus_{j}\shQ_{A_{j}/R^{s}}\text{ and }f^{A^{s}/R^{s}}=\sum_{j}f^{A_{j}/R_{s}}
\]
Finally we have that the following three conditions are equivalent:
$A$ is regular; $A^{s}$ is regular; $A_{j}$ is regular for all
$j=1,\dots,q$.\end{rem}
\begin{proof}
(of \ref{main}) By \ref{rem: first considerations on traces} and
\ref{rem: passing to the strict henselization} we can assume that
$R$ is strictly Henselian, that $A$ is local, so that $|\Spec(A\otimes_{R}\overline{k_{R}})|=1$,
and that $A/R$ is generically étale. Moreover in this case the following
three conditions are equivalent by \ref{cor:tameness and h} and \ref{lem:when tr is surjective}:
$A/R$ is tame with separable residue fields; $k_{A}=k_{R}$ and $\rk A\in R^{*}$;
$\tr_{A/R}\colon A\arr R$ is surjective.

\emph{Inequality and }$1)\iff3)$. By \ref{rem: first considerations on traces}
we can assume that $\tr_{A}\colon A\arr R$ is surjective, so that
$k_{A}=k_{R}$, $\rk A\in R^{*}$. By \ref{cor:ker of the trace}
we also have $\Ker\tr_{A}\subseteq m_{A}$ and $\tr_{A}(m_{A})\subseteq m_{R}$.
Using \ref{rem: first considerations on traces} and its notation,
we see that $f^{A}=v_{R}(\det N)$. If $i,j\geq1$ then $x_{i}x_{j}\in m_{A}$
and thus $\tr_{A}(x_{i}x_{j})\in m_{R}$, that is all entries of $N$
belongs to $m_{R}$. If $\pi\in m_{R}$ is an uniformizer of $R$
we can write $N=\pi N'$ where $N'$ is a matrix with entries in $R$.
In particular
\[
\det N=\pi^{\rk A-1}\det N'
\]
Applying the valuation of $R$ we get $f^{A}=v_{R}(\det N)=\rk A-1+v_{R}(\det N')$
and the desired inequality. 

If $\shQ_{A}\simeq\Coker(N)$ is defined over $k_{R}$ we obtain a
surjective map $k_{R}^{\rk A-1}\arr\shQ_{A}\otimes k_{R}\simeq\shQ_{A}$
and therefore that $f^{A}=\l(\shQ_{A})\leq\rk A-1$. Conversely, if
$f^{A}=\rk A-1$, which means that $N'\colon R^{\rk A-1}\arr R^{\rk A-1}$
is an isomorphism, we have $\shQ_{A}\simeq\Coker(\pi N')$, so that
$\shQ_{A}\simeq k_{R}^{\rk A-1}$ is defined over $k_{R}$ as required.

$2)\then1)$ Let $t\in A$ be a generator of the maximal ideal. The
$k_{R}$-algebra $A\otimes k_{R}$ is local, with residue field $k_{R}$
and its maximal ideal is generated by the projection $\overline{t}$
of $t$. It is easy to conclude that $A\otimes k_{R}=k_{R}[X](X^{N})$
where $N=\rk A$ and $X$ corresponds to $\overline{t}$. By Nakayama's
lemma it follows that $1,t,\dots,t^{N-1}$ is an $R$-basis of $A$
and thus that $A\simeq R[Y]/(Y^{N}-g(Y))$ where $Y$ corresponds
to $t$, $\deg g<N$ and all coefficients of $g$ are in $m_{R}$.
Since $m_{A}=\langle t,m_{R}\rangle_{A}$, the condition that $A$
is regular tells us that $v_{R}(g(0))=1$. We are going to compute
the valuation of the determinant of $\widetilde{\tr_{A}}\colon A\arr\duale A$
writing this map in terms of the basis $1,t,\dots,t^{N-1}$, that
is the valuation of the determinant of the matrix $(\tr_{A}(t^{i+j}))_{0\leq i,j<N}$
(see \ref{rem: first considerations on traces}). Set $q_{s}=\tr_{A}(t^{s})$.
By \ref{cor:ker of the trace} or a direct computation we know that
$q_{S}\in m_{R}$ for $s>0$. In particular, since $v_{R}(g(0))=1$,
we also have $v_{R}(q_{N})=1$. Moreover it follows by induction that
$v_{R}(q_{s})>1$ if $s>N$. Set $S_{N}$ for the group of permutations
of the set $\{0,1,\dots,N-1\}$. We have
\[
\det((\tr_{A}(t^{i+j}))_{0\leq i,j<N})=\sum_{\sigma\in S_{N}}(-1)^{\text{sgn}(\sigma)}z_{\sigma}\text{ where }z_{\sigma}=q_{0+\sigma(0)}q_{1+\sigma(1)}\cdots q_{N-1+\sigma(N-1)}
\]
We claim that $v_{R}(z_{\sigma})\geq N$ for all $\sigma\in S_{n}$
but the permutation $\sigma(0)=0$ and $\sigma(i)=N-i$ for $i\neq0$,
for which $v_{R}(z_{\sigma})=N-1$. This will conclude the proof.
Let $\sigma\in S_{N}$. If $\sigma(0)\neq0$, then all the $N$-factors
of $z_{\sigma}$ are in $m_{R}$, which implies that $v_{R}(z_{\sigma})\geq N$.
Thus assume that $\sigma(0)=0$. Since $q_{0}=\tr_{A}(1)=\rk A\in R^{*}$
we see that $z_{\sigma}$ is, up to $q_{0}$, the product of $N-1$
elements of $m_{R}$. If one of those factors is of the form $q_{i+\sigma(i)}$
with $i+\sigma(i)>N$ then $v_{R}(z_{\sigma})\geq N$ because $v_{R}(q_{s})\geq2$
if $s>N$. Thus the only case left is when $\sigma(i)\leq N-i$ for
all $0<i<N$ and $\sigma(0)=0$. But, arguing by induction, this $\sigma$
is unique and it is given by $\sigma(0)=0$ and $\sigma(i)=N-i$.
In this case we have
\[
z_{\sigma}=\rk A(\tr(t^{N}))^{N-1}
\]
and therefore $v_{R}(z_{\sigma})=N-1$ as required. 

$1)\then2)$ Since $1)\iff3)$, we already know that $A/R$ is tame
with separable residue fields. We have to prove that $A$ is regular.
Set $k(R)$ for the fraction field of $R$. Since $A\otimes k(R)$
is etale over $k(R)$, it is a product of fields $L_{1},\dots,L_{s}$
which are separable extensions of $k(R)$. Let $B$ be the integral
closure of $R$ inside $A\otimes k(R)$. We have that $A\subseteq B$
and that $B=B_{1}\times\cdots\times B_{s}$ where $B_{i}$ is the
integral closure of $R$ inside $L_{i}$. Since $R$ is strictly Henselian
and the $B_{i}$ are domains, it follows that they are local. Moreover
since $R$ is a DVR we can also conclude that the $B_{i}$ are DVR.
We are going to prove that $s=1$ and $A=B$. Notice that $\rk B=\rk A$
and denote by $j\colon A\arr B$ the inclusion. By computing $\tr_{A}$
and $\tr_{B}$ over $A\otimes k(R)\simeq B\otimes k(R)$ we can conclude
that $(\tr_{B})_{|A}=\tr_{A}$. In particular we obtain a commutative
diagram of free $R$-modules   \[   \begin{tikzpicture}[xscale=1.5,yscale=-1.2]     \node (A0_0) at (0, 0) {$A$};     \node (A0_1) at (1, 0) {$\duale A$};     \node (A1_0) at (0, 1) {$B$};     \node (A1_1) at (1, 1) {$\duale B$};     \path (A0_0) edge [->]node [auto] {$\scriptstyle{\widetilde{ \tr_A}}$} (A0_1);     \path (A0_0) edge [->]node [auto] {$\scriptstyle{j}$} (A1_0);     \path (A1_0) edge [->]node [auto] {$\scriptstyle{\widetilde{ \tr_B}}$} (A1_1);     \path (A1_1) edge [->]node [auto] {$\scriptstyle{\duale j}$} (A0_1);   \end{tikzpicture}   \] Notice
that $\det j=\det\duale j$. Thus taking determinants and then valuations
we obtain the expression
\[
f^{A}=2v_{R}(\det j)+f^{B}=2v_{R}(\det j)+\sum_{i=1}^{s}f^{B_{i}}
\]
Since $k$ is separably closed and thanks to \ref{cor:Length of geometric fiber}
we have that $|\Spec(B_{i}\otimes_{R}\overline{k})|=1$. In particular
$f^{B_{i}}\geq\rk B_{i}-1$ by the inequality in the statement. Since
$f^{A}=\rk A-1$ we get
\[
s\geq1+2v_{R}(\det j)
\]
We are going to prove that $v_{R}(\det j)\geq s-1$. This will end
the proof because it implies $s=1$ and $v_{R}(\det j)=0$, that is
$B=A$. Since $k_{A}=k_{R}$ we have
\[
v_{R}(\det j)=\l(B/A)\geq\dim_{k_{R}}(B/(A+m_{A}B))
\]
Denote by $e_{1},\dots,e_{s}\in B$ the idempotents corresponding
to the decomposition $B=B_{1}\times\cdots\times B_{s}$. We will prove
that $e_{2},\dots,e_{s}$ are $k_{R}$-linearly independent in $B/(A+m_{A}B)$,
that is we prove that if 
\[
x=x_{1}e_{1}+\cdots+x_{s}e_{s}\in A+m_{A}B\text{ where }x_{1}=0\text{ and }x_{i}\in R
\]
 then all $x_{i}$ are in the maximal ideal $m_{R}$. Notice that,
since $k_{A}=k_{R}$, $1$ and $m_{A}$ generates $A$ as $R$-module.
In particular $A+m_{A}B$ is generated by $1$ and $m_{A}B$ as $R$-module.
Moreover since the maps $A\arr B\arr B_{i}$ map $m_{A}$ inside $m_{B_{i}}$
it follows that $m_{A}B\subseteq m_{B_{1}}\times\cdots\times m_{B_{s}}$.
Thus $x$ can be written as 
\[
x=\alpha+c_{1}e_{1}+\cdots+c_{s}e_{s}\text{ with }\alpha\in R\text{ and }c_{i}\in m_{B_{i}}
\]
inside $B$. In particular $0=x_{1}=\alpha+c_{1}$ in $B_{1}$, which
implies that $\alpha\in m_{B_{1}}\cap R=m_{R}$. Thus if $i>0$ we
have
\[
x_{i}=\alpha+c_{i}\in m_{B_{i}}\cap R=m_{R}
\]
as required.
\end{proof}
We show via some examples how the conditions of tameness and separability
of residue fields in \ref{main} cannot be omitted. %

\begin{example}
Let $R=\Z_{2}$ be the ring of $2$-adic numbers and consider the
$R$-algebra
\[
A=\frac{R[X]}{(X^{2}-c)}\text{ with }c\in R
\]
We have that $\tr_{A}(X)=0$ and $\tr_{A}(X^{2})=2c$, so that the
matrix of $\widetilde{\tr_{A}}\colon A\arr A^{\vee}$ is
\[
M=\left(\begin{array}{cc}
2 & 0\\
0 & 2c
\end{array}\right)
\]
In particular $f^{A}=v_{R}(\det M)\geq2>\rk A-|\Spec(A\otimes_{R}\overline{k_{R}})|$
and $\shQ_{A}$ is defined over $\F_{2}$ if and only if $c\in R^{*}$. 

Assume $c=2t+d^{2}$ with $t,d\in R$, so that $A/m_{R}A\simeq\F_{2}[Y]/(Y^{2})$,
$A$ is local with maximal ideal $(m_{R},X-d)$, has separable residue
fields and it is not tame. We see that $\shQ_{A}$ is defined over
$\F_{2}$ if and only if $d\in R^{*}$, while $A$ is regular if and
only if $t\in R^{*}$. 

Assume $c\in R^{*}$ and that $c$ is not a square in $\F_{2}$. In
this case $A$ is local with maximal ideal $m_{R}A$, $A/R$ is tame,
its residue field is not separable, it is regular and $\shQ_{A}$
is defined over $\F_{2}$.

\end{example}

\end{document}